\documentclass[11pt, a4paper]{article}
\usepackage{times}
\usepackage{hyperref}
\usepackage[british]{babel}
\usepackage{enumerate, longtable}
\usepackage{amsmath, amscd, amsfonts, amsthm, amssymb, latexsym, comment, stmaryrd, graphicx, color}
\usepackage[all]{xypic}
\usepackage[T1]{fontenc}
\usepackage[utf8]{inputenc}
\usepackage{fullpage}
\usepackage{url}

\newtheorem{thm}{Theorem}[section]
\newtheorem{lem}[thm]{Lemma}
\newtheorem{defi}[thm]{Definition}
\newtheorem{rem}[thm]{Remark}

\newtheorem{prop}[thm]{Proposition}
\newtheorem{cor}[thm]{Corollary}

\newtheorem{problem}[thm]{Problem}

\newcommand{\GL}{\mathrm{GL}}
\newcommand{\SL}{\mathrm{SL}}
\newcommand{\PGL}{\mathrm{PGL}}

\newcommand{\proj}{\mathrm{proj}}

\DeclareMathOperator{\Image}{im}
\DeclareMathOperator{\Gal}{Gal}

\DeclareMathOperator{\Frob}{Frob}

\newcommand{\cG}{\mathcal{G}}

\newcommand{\cP}{\mathcal{P}}

\newcommand{\CC}{\mathbb{C}}
\newcommand{\Cc}{\mathbb{C}}

\newcommand{\FF}{\mathbb{F}}
\newcommand{\Ff}{\mathbb{F}}

\newcommand{\QQ}{\mathbb{Q}}
\newcommand{\Qq}{\mathbb{Q}}

\newcommand{\ZZ}{\mathbb{Z}}
\newcommand{\Zz}{\mathbb{Z}}
\newcommand{\Qbar}{\overline{\QQ}}
\newcommand{\Zbar}{\overline{\ZZ}}

\newcommand{\rhobar}{{\overline{\rho}}}

\newcommand{\GQ}{{{\rm{G}}_\Qq}}

\begin{document}

\title{Galois families of modular forms and application to weight one}

\author{Sara Arias-de-Reyna, Fran\c{c}ois Legrand, Gabor Wiese}

\maketitle

\begin{abstract}
We introduce Galois families of modular forms. They are a new kind of family coming from Galois representations of the absolute Galois groups of rational function fields over $\Qq$. We exhibit some examples and provide an infinite Galois family of non-liftable weight one Katz modular eigenforms over $\overline{\Ff_p}$ for $p \in \{3,5,7,11\}$.

MSC Classification: 11F80 (Galois representations), 11F11 (Holomorphic modular forms of integral weight), 12F12 (Inverse Galois Theory), 12E30 (Field Arithmetic).
\end{abstract}

\section{Introduction} \label{sec:intro}

Families of modular forms and their attached Galois representations are of fundamental importance in current arithmetic geometric research. With this paper, we would like to draw attention to a new kind of families of modular forms, which we call {\em Galois families} (cf.~Definition~\ref{defi:GF}). For example, projective Galois families are defined as follows:

\vspace{2mm}

\noindent
{\bf{Definition A.}} {\it{Let $(f_i)_{i \in I}$ be a family of normalised Hecke eigenforms (of any level and weight). For a prime number $p$ and a finite subgroup $G \subset {\rm{PGL}}_2(\overline{\Ff_p})$, we say that the $(f_i)_{i \in I}$ form a {\rm{projective $G$-Galois family}} if the following two conditions hold:

\vspace{0.5mm}

\noindent
{\rm{(1)}} for each $i \in I$, the image of the projective mod $p$ Galois representation $\rho^\proj_{f_i, p} : {\rm{Gal}}(\overline{\Qq}/\Qq) \rightarrow {\rm{PGL}}_2(\overline{\Ff_p})$ associated with $f_i$ is conjugate to $G$,

\vspace{0.5mm}

\noindent
{\rm{(2)}} there exists a finite Galois extension $E$ of a rational function field $\Qq({\bf{T}}) = \Qq(T_1, \dots, T_n)$ with Galois group $G$ such that, for each $i \in I$, the number field $K^\proj_{f_i,p}$ 'cut out by $\rho^\proj_{f_i, p}$', that is, defined by ${\rm{ker}}(\rho^\proj_{f_i,p}) = {\rm{Gal}}(\overline{\Qq}/K^\proj_{f_i,p})$, is obtained by specialising the function field $E$ at some ${\bf{t}}_i \in \Qq^n$.}}

\vspace{2mm}

\noindent
We refer to \S\ref{sec:bas} for standard terminology and material on Galois representations, modular forms and functions field extensions.

Galois families are then taken with respect to a prime number~$p$ and have as feature that the projective mod~$p$ Galois representations of all members of the family have conjugate images. The family can be taken to consist of 'classical' holomorphic Hecke eigenforms or it can be chosen to be made of Katz modular Hecke eigenforms (geometrically) defined over $\overline{\Ff_p}$. We furthermore define projective Artin Galois families consisting of holomorphic Hecke eigenforms of weight one in a similar way, and also consider the linear case (in both settings). See Definition~\ref{defi:GF} for more details.

Galois families are fundamentally different from other kinds of families of modular forms, such as Hida families (see, e.g.,~\cite{Em11} for an account of the theory). On the one hand, if we took the classical members of a Hida family, the field cut out by the mod $p$ Galois representation would be the same in all cases. On the other hand, Galois families do not see $p$-adic deformations, so they miss the interesting information in Hida families. Galois families are rooted in field arithmetic and we see our paper as a step towards strengthening connections between field arithmetic and the automorphic theory.

In \S\ref{sec:BB}, we relate our notion of Galois families to problems and results from field arithmetic, such as the Beckmann-Black Problem and the existence of parametric extensions over $\Qq$ (recalled as Problem \ref{prob:BB} and Definition \ref{def:par_gen}, respectively), and formulate analogues for modular forms (see Problem \ref{prob:mf} and Definition \ref{def:para_mf}), which we partially answer. For example, if $G$ is any of the three groups $A_4$, $S_4$, $A_5$, by using results on the existence or non-existence of generic or parametric polynomials/extensions with rational coefficients for the group $G$, we prove that the family of all holomorphic normalised Hecke eigenforms of weight one with 'exceptional' projective Galois image $G$ is a projective Artin $G$-Galois family and that 2 is the minimal number of parameters we need to get such a family. See Theorems \ref{thm:parametric_poly2} and \ref{?} for more details.

In \S\ref{sec:explicit}, we focus on some Galois families consisting of modular forms of weight one since those are special in a number of ways. For example, a weight one Hecke eigenform which is geometrically defined over $\overline{\Ff_p}$ in the sense of Katz need not lift to a holomorphic weight one eigenform. Such non-liftable Hecke eigenforms are torsion classes in the cohomology of the relevant modular curve over $\overline{\Ff_p}$. As such they are sporadic. This is one, but not the only reason that we do not have any closed formulas for the dimension of spaces of weight one modular forms, not even over the complex numbers (another reason is that holomorphic weight one Hecke eigenforms with `exceptional' projective Galois image as above also seem to occur sporadically).

To the best of our knowledge, it is not known whether, for a given prime number $p$, there are infinitely many non-liftable Hecke eigenforms of weight one over $\overline{\Ff_p}$ with pairwise non-isomorphic projective Galois representations. We provide a general criterion which reduces to the existence of a finite Galois extension of the rational function field $\Qq(T)$ with specified Galois group admitting specialisations with specified local behaviour at some rational places, including the one associated with the prime number $p$ (see Theorem \ref{thm1}). We then use standard results on the local behaviour of specialisations (recalled in \S\ref{ssec:ffe}) and the existence of some explicit bivariate polynomials with rational coefficients to construct such infinite families for $p \in \{3,5,7,11\}$:

\vspace{2mm}

\noindent
{\bf{Theorem B.}} {\it{For every $p \in \{3,5,7,11\}$, there exists a finite group $G$ among ${\rm{PGL}}_2(\mathbb{F}_{p})$, ${\rm{PSL}}_2(\mathbb{F}_{p})$ and ${\rm{PSL}}_2(\mathbb{F}_{p^2})$ which fulfills the following: there exists an infinite projective $G$-Galois family consisting of Katz modular forms of weight one such that no family member is liftable to a holomorphic weight one Hecke eigenform in any level.}}

\vspace{2mm}

\noindent
See Corollaries \ref{coro4} and \ref{cor5} for more precise results, where the corresponding group $G$ is explicitly given.

\subsection*{Acknowledgements}

The authors would like to thank Lior Bary-Soroker, Pierre D\`ebes and Ian Kiming for providing useful feedback on early versions of the article.

S.~Arias-de-Reyna was supported by projects US-1262169 (Junta de Andalucía and FEDER, UE) and MTM2016-75027-P (Ministerio de Economía y Competitividad, Spain).

\section{Basics} \label{sec:bas}

The aim of this section is to present the standard material on Galois representations, modular forms and function field extensions that will be used throughout the present article.

\subsection{Galois representations and modular forms} \label{ssec:gr}

In this article, we shall be concerned with two-dimensional Galois representations which are either Artin, i.e., are defined over $\CC$, or have coefficients in $\overline{\Ff_p}$, where $p$ is a prime number. Both cases will be treated in parallel. To this end, we let 
$$\cG \in \{\GL_2(\overline{\Ff_p}), \GL_2(\CC)\}.$$
Then all Galois representations considered in this paper are continuous and of the form
$$\rho: \GQ \to \cG,$$
where $\GQ = \Gal(\Qbar/\QQ)$ is the absolute Galois group of $\QQ$. The image $\Image(\rho) \subset \cG$ is a finite group. By standard Galois theory,  one has $\Image(\rho) \cong \Gal(K_\rho/\QQ)$ for some number field~$K_\rho$. We refer to $K_\rho$ as the `number field cut out by $\rho$'. It can also be characterised by $\ker(\rho) = \Gal(\Qbar/K_\rho)$.

A Galois representation $\rho$ as above is said to be {\em irreducible} if the underlying $\overline{\Ff_p}[\GQ]$-module (or $\CC[\GQ]$-module) is irreducible. It is said to be {\em semi-simple} if it is a direct sum of simple modules. Accordingly, a finite subgroup $G \subset \cG$ is called {\em irreducible} if the natural inclusion is an irreducible representation. One furthermore says that $\rho$ is {\em odd} if $\det(\rho(c))=-1$, where $c$ is any complex conjugation in $\GQ$ (all are conjugate). For a prime number $\ell$, one says that $\rho$ is {\em unramified at~$\ell$} if the inertia group at~$\ell$ inside $\Gal(\overline{\Qq_\ell}/\QQ_\ell) \hookrightarrow \GQ$ (for any embedding) lies in the kernel of~$\rho$. This is equivalent to $K_\rho$ being unramified above~$\ell$.
 
We also consider {\em projective} Galois representations in the two cases. We accordingly let
$$\cP \in \{\PGL_2(\overline{\Ff_p}), \PGL_2(\CC)\},$$
consider
$$\rho^\proj:\GQ  \to \cP$$
and make similar definitions, such as $\Image(\rho^\proj) \cong \Gal(K_{\rho^\proj}/\QQ)$. Given a Galois representation $\rho$, one can associate to it a unique projective one $\rho^\proj$ via composition with the natural projection $\cG \twoheadrightarrow \cP$. If we set $G=\Image(\rho)$ and $G^\proj=\Image(\rho^\proj)$, then $G^\proj$ is the image of $G$ under the natural projection. Denote by $H$ the kernel of $G \to G^\proj$ and let $\rhobar : {\rm{G}}_\Qq / {\rm{ker}}(\rho) \rightarrow G$ be the isomorphism induced by $\rho$. If we identify $\Gal(K_\rho/\QQ)$ and $\GQ / {\rm{ker}}(\rho)$ (via restriction of automorphisms), then we have $K_{\rho^\proj} = (K_\rho)^{\rhobar^{-1}(H)}$ by standard Galois theory. We shall sometimes simply see $H$ as a subgroup of $\Gal(K_\rho/\QQ)$, i.e., drop $\rhobar^{-1}$ from the notation, but we must be aware of the dependence on the representation.

Given a projective Galois representation $\rho^\proj$ as above, by a result of Tate (see \cite[\S6]{Ser77} and \cite[\S4]{Que95}), it can be lifted to a linear representation $\rho$ as above. Moreover, it can be ensured that $\rho$ is unramified at all prime numbers where $\rho^\proj$ is unramified.

\medskip

This article relies on certain kinds of modular forms, which are called Hecke eigenforms. One can attach Galois representations of the above kinds to them. The modular forms we consider are either the `classical' holomorphic modular forms defined in standard textbooks such as \cite{DS05}, or their geometric counter part due to Katz \cite{Kat73}. A good source for both is \cite{DI95}. In both settings, modular forms have a weight, an integer usually denoted~$k$, and a level, a congruence subgroup of $\SL_2(\ZZ)$.
We shall exclusively work with levels $\Gamma_1(N)$ and usually just say that the positive integer $N$ is the level. Holomorphic modular forms of fixed weight and level form a finite dimensional $\CC$-vector space. Katz' geometric definition allows the use of other base rings, provided the level is invertible in the ring. We shall use only $\CC$ or $\overline{\Ff_p}$ as base fields and thus impose $p \nmid N$ in the latter case. It should be remarked that using Katz modular forms over~$\CC$, one exactly recovers classical modular forms.

Every modular form has a so-called {\em $q$-expansion}, that is, a power series $\sum_{n=0}^\infty a_n q^n$ with $a_n$ in the base field. If the base field is $\CC$, then replacing $q$ by $e^{2\pi i z}$, one obtains a Fourier series, which is actually equal to the modular form, viewed as a holomorphic function on the upper half-plane. When working over~$\CC$, one can consider the $\ZZ$-module of $q$-expansions of modular forms in fixed weight~$k$ and level~$N$ such that all $a_n$ lie in~$\ZZ$. One can reduce them modulo~$p$ and base extend to~$\overline{\Ff_p}$. If $k \ge 2$, one then essentially recovers Katz modular forms over~$\overline{\Ff_p}$ (see \cite[Lemma 1.9]{Edi97} for a precise statement). However, when $k=1$, there are often more Katz modular forms than reductions of holomorphic ones. A purpose of this article is to exhibit infinite families of such having interesting properties.

In both settings, there is a family of commuting linear maps $T_m$, for $m \in \ZZ_{\ge 1}$, on the vector space of modular forms, called {\em Hecke operators}. A modular form that is an eigenform for each $T_m$ is called a {\em Hecke eigenform}. If, moreover, the coefficient $a_1$ in the $q$-expansion equals~$1$, then call the eigenform {\em normalised}. If a normalised Hecke eigenform is a Katz modular form over~$\overline{\Ff_p}$, then the $a_n$ are, of course, in $\overline{\Ff_p}$. If it is a holomorphic modular form, then the eigen-property implies that the $a_n$ are algebraic integers, that is, lie in $\Zbar$, the integral closure of $\ZZ$ in~$\CC$. For the entire article, we fix ring homomorphisms $\Zbar \hookrightarrow \overline{\Zz_p} \twoheadrightarrow \overline{\Ff_p}$. By the {\em reduction modulo~$p$} of an algebraic integer, we shall always understand the image under the composite maps. So, the coefficients $a_n$ of a normalised holomorphic Hecke eigenform have well-defined reductions modulo~$p$.

\medskip

Work of Shimura, Deligne and Deligne--Serre (see \cite[\S9.6]{DS05} and \cite{DS75}) attaches to any normalised Hecke eigenform $f = \sum_{n=0}^\infty a_n q^n$ of weight $k$ and level~$N$ a semi-simple Galois representation $\rho_{f,p}: \GQ \to \GL_2(\overline{\Ff_p})$. The representation is known to be odd and unramified outside~$Np$. Moreover, at every prime $\ell \nmid Np$, the trace of any Frobenius element $\Frob_\ell$ (all are conjugate) equals (the reduction modulo~$p$ of)~$a_\ell$ and its determinant equals $\ell^{k-1} \epsilon(\ell)$, where $\epsilon$ is the nebentype character associated with~$f$ (its values are the eigenvalues for the action of the diamond operators on~$f$). If $f$ is a holomorphic normalised Hecke eigenform (over $\CC$) of weight $k=1$, by Deligne--Serre, one can associate with it a semi-simple Artin representation $\rho_{f,\CC}: \GQ \to \GL_2(\CC)$, which is odd, unramified outside~$N$ and, for every prime $\ell \nmid N$, the trace of $\Frob_\ell$ equals $a_\ell$ (as complex numbers) and its determinant equals $\epsilon(\ell)$. 

The projectivisation of the representations will be denoted $\rho_{f,p}^\proj$ and $\rho_{f,\CC}^\proj$, respectively. As abbreviations, we shall often write $\rho_f$ and $\rho_f^\proj$ for both the mod~$p$ and the Artin cases. Moreover, we set $K_f := K_{\rho_f}$ and $K_f^\proj := K_{\rho_{f}^\proj}$. If $p>2$, the oddness of $\rho_f$ implies that $K_f$ and $K_f^\proj$ are totally imaginary. Indeed, as complex conjugation is of determinant~$-1$, it is a non-trivial involution and not scalar. We point out that the representations $\rho_f$ need not be irreducible. However, if they are, then the underlying modular form is {\em cuspidal}. For $k=1$, $\rho_{f, \CC}$ is irreducible if and only if $f$ is cuspidal.

We recall that a theorem of Brauer and Nesbitt (\cite[30.16]{CR88}) states that a semi-simple representation of a finite group is uniquely determined up to isomorphism by its character. Thus, any semi-simple Galois representation with finite image is uniquely determined by the traces of the Frobenius elements $\Frob_\ell$ at primes $\ell$ in a set of primes of density one because by Chebotarev's density theorem \cite[VII.13.4]{Neu99} any element in the image of the representation comes from some $\Frob_\ell$ (in fact, for $\ell$ in a positive density set of primes).
As the trace of $\rho_f(\Frob_\ell)$ equals the coefficient $a_\ell$ of~$f$, the semi-simple representation $\rho_f$ is hence uniquely determined by~$f$ up to isomorphism of Galois representations, that is, up to conjugation. Consequently, the images of $\rho_f$ and $\rho_f^\proj$ are uniquely determined by $f$ up to conjugation in $\cG \in \{\GL_2(\overline{\Ff_p}), \GL_2(\CC)\}$ or $\cP  \in \{\PGL_2(\overline{\Ff_p}), \PGL_2(\CC)\}$.

A very important theorem of Khare-Wintenberger and Kisin (\cite[Theorem 1.2]{KW09} and \cite[Corollary 0.2]{Kis09}) is the following, which was formely known as {\em Serre's Modularity Conjecture}.

\begin{thm}\label{thm:Serre-Conj}
Let $p$ be a prime number and $\rho: \GQ \to \GL_2(\overline{\Ff_p})$ an odd irreducible Galois representation. Then there is a normalised Hecke eigenform (of some level and weight) such that $\rho \cong \rho_{f,p}$.
\end{thm}

A notable consequence is the modularity of Artin representations of the following type (\cite[Corollary 10.2(ii)]{KW09}).

\begin{thm}\label{thm:Serre-Artin}
Let $\rho: \GQ \to \GL_2(\CC)$ be an odd and irreducible Galois representation. Then there is a normalised `classical' holomorphic Hecke eigenform of weight one (and some level) such that $\rho \cong \rho_{f,\CC}$.
\end{thm}

In view of our desire to make elegant and short statements, we make the following convention. If a representation $\rho: \GQ \to \cG$ (resp., a projective represention $\rho^\proj: \GQ \to \cP$) comes from a normalised Hecke eigenform~$f$, then we assume $f$ to be holomorphic of weight one if we are in the Artin case.

The following practical consequence of the above shall be used on several occasions in the sequel:

\begin{prop} \label{tool0}
Let $G \subset \cG$ (resp., $G \subset \cP$) be a finite irreducible subgroup and $F/\Qq$ a Galois extension of group $G$. Then there exists a normalised Hecke eigenform $f$ such that $K_{f}=F$ (resp., $K_f^\proj = F$) if and only if 

\vspace{0.5mm}

\noindent
- $F^{\text{scalars in } G}$ (resp., $F$) is totally imaginary if we are in the Artin case or in the mod $p$ case with $p\ge 3$,

\vspace{0.5mm}

\noindent
- $F$ is arbitrary if we are in the mod $p$ case with $p=2$.
\end{prop}

\begin{proof}
We prove this statement in the mod $p$ case. The arguments in the Artin case are exactly the same, except that one has to invoke the modularity of odd and irreducible Artin representations from Theorem~\ref{thm:Serre-Artin}.

First, assume there exists a normalised Hecke eigenform $f$ such that $K_{f}=F$. Then, as recalled above, $\rho_{f}$ is odd. If $p \geq 3$, this implies that $K_{f}^{\proj}=F^{\text{scalars in }G}$ is totally imaginary. 
Now, assume $F^{\text{scalars in }G}$is totally imaginary if $p \geq 3$. We view the extension $F/\Qq$ as a Galois representation
$$\rho : \GQ \twoheadrightarrow \Gal(F/\QQ) \cong G \subset \cG.$$
Then $\rho$ is odd. Note that oddness is an empty condition if $p=2$, so it suffices to consider $p\ge 3$. Then, indeed, as $F^{\text{scalars in }G}$ is totally imaginary, (any) complex conjugation is sent to a non-scalar element in $G$. Hence, under $\rho$, complex conjugation maps to a non-scalar involution in $\cG$ and thus has determinant $-1$. As $\rho$ is also irreducible, it is afforded by a normalised Hecke eigenform $f$ by Theorem~\ref{thm:Serre-Conj}.

\vspace{0.5mm}

\noindent
In the projective case, the same arguments as above yield the desired equality $K_f^\proj = F$, except that one has to invoke Tate's theorem (recalled above) to lift the projective representation to a linear one to obtain the modularity in the last step.
\end{proof}

\subsection{Function field extensions} \label{ssec:ffe}

Given a field $k$ of characteristic zero, an integer $n \geq 1$ and an $n$-tuple ${\bf{T}}=(T_1, \dots,T_n)$ of algebraically independent indeterminates, let $E/k({\bf{T}})$ be a finite Galois extension. If $n=1$, we write $E/k(T)$ for simplicity. Say that $E/k({\bf{T}})$ is {\it{$k$-regular}} if $E \cap \overline{k}= k$.

Let ${B}$ be the integral closure of $k[{\bf{T}}]$ in ${E}$. For ${\bf{t}}=(t_1, \dots, t_n) \in k^n$, the residue field of ${B}$ at a maximal ideal $\mathfrak{P}$ lying over the ideal $\langle {\bf{T}} - {\bf{t}} \rangle$ of $k[{\bf{T}}]$ generated by $T_1-t_1, \dots, T_n-t_n$ is denoted by ${E}_{\bf{t}}$ and the extension ${E}_{\bf{t}}/k$ is called the {\it{specialisation}} of ${E}/k({\bf{T}})$ at ${\bf{t}}$. As the extension ${E}/k({\bf{T}})$ is Galois, the field ${E}_{\bf{t}}$ does not depend on $\mathfrak{P}$ and the extension ${E}_{\bf{t}}/k$ is finite and Galois. Moreover,  the Galois group of ${E}_{\bf{t}}/k$ is the quotient of the decomposition group of $E/k({\bf{T}})$ at $\mathfrak{P}$ by the inertia group at $\mathfrak{P}$. For ${\bf{t}}$ outside a Zariski-closed proper subset (depending only on ${E}/k({\bf{T}})$), the inertia group at $\mathfrak{P}$ is trivial; in particular, the Galois group of ${E}_{\bf{t}}/k$ is a subgroup of ${\rm{Gal}}({E}/k({\bf{T}}))$. Furthermore, if $P({\bf{T}},Y) \in k[{\bf{T}}][Y]$ is a monic separable polynomial of splitting field $E$ over $k({\bf{T}})$ and if ${\bf{t}} \in k^n$ is such that the splitting field of $P({\bf{t}},Y)$ has Galois group ${\rm{Gal}}(E/k({\bf{T}}))$ over $k$, then the field $E_{\bf{t}}$ is the splitting field over $k$ of $P({\bf{t}},Y)$ \footnote{Indeed, since $P({\bf{T}},Y)$ is monic and is in $k[{\bf{T}}][Y]$, the splitting field over $k$ of $P({\bf{t}},Y)$ is contained in the specialised field $E_{\bf{t}}$. As the former field has degree $|G|$ over $k$ and the latter has degree at most $|G|$ over $k$, the two fields coincide.}.

Assume $n=1$. A point $t_0 \in  \mathbb{P}^1(\overline{k})$ is a {\it{branch point}} of $E/k(T)$ if the prime ideal of $\overline{k}[T-t_0]$ generated by $T-t_0$ ramifies in the extension $E\overline{k}/\overline{k}(T)$ \footnote{Replace $T-t_0$ by $1/T$ if $t_0=\infty$.}. The extension $E/k(T)$ has only finitely many branch points, usually denoted by $t_1, \dots, t_r$, and one has $r=0$ if and only if $E\overline{k}=\overline{k}(T)$ (which is equivalent to $E=k(T)$ if $E/k(T)$ is $k$-regular). If $t_0 \in k \setminus \{t_1, \dots,t_r\}$, then the Galois group of the specialisation $E_{t_0}/k$ of $E/k(T)$ at $t_0$ is the decomposition group at a prime ideal $\mathfrak{P}$ lying over $\langle T-t_0 \rangle$. Moreover, if $E$ is the splitting field over $k(T)$ of a monic separable polynomial $P(T,Y) \in k[T][Y]$ and if $t_0$ is any element of $k$ such that $P(t_0,Y)$ is separable, then $t_0$ is not a branch point of $E/k(T)$ and the field $E_{t_0}$ is the splitting field over $k$ of $P(t_0,Y)$.

Let $E/\Qq(T)$ be a $\Qq$-regular Galois extension. In the sequel, we shall deal with the local behaviour at prime numbers of specialisations of $E/\Qq(T)$. This requires the following material. See \cite[\S2.2]{Leg16} for more details.

Given a number field $F$, let $A$ be the ring of its integers. For a non-zero prime ideal $\mathfrak{P}$ of $A$, we denote the corresponding valuation of $F$ by $v_\mathfrak{P}$. Moreover, we identify $\mathbb{P}^1(F)$ with $F \cup \{\infty\}$ and set $1/\infty = 0$, $1 / 0 = \infty$, $v_\mathfrak{P}(\infty) = -\infty$ and $v_\mathfrak{P}(0) = \infty$.

\begin{defi} \label{rencontre} 
{\rm{(1)}} Let $F$ be a number field, $A$ the ring of its integers, $\mathfrak{P}$ a (non-zero) prime ideal of $A$ and $t_0$, $t_1 \in \mathbb{P}^1(F)$. We say that {\em{$t_0$ and $t_1$ meet modulo $\mathfrak{P}$}} if one of the following conditions holds: 

\vspace{0.5mm}

{\rm{(a)}} $v_{\mathfrak{P}}(t_0) \geq 0$, $v_{\mathfrak{P}}(t_1) \geq 0$ and $v_{\mathfrak{P}}(t_0-t_1) > 0$,

\vspace{0.5mm}

{\rm{(b)}} $v_{\mathfrak{P}}(t_0) \leq 0$, $v_{\mathfrak{P}}(t_1) \leq 0$ and $v_{\mathfrak{P}}((1/t_0) - (1/t_1)) > 0$.

\vspace{1mm}

\noindent
{\rm{(2)}} Given $t_0$, $t_1$ $\in \mathbb{P}^1(\overline{\Qq})$ and a prime number $p$, we say that {\em{$t_0$ and $t_1$ meet modulo $p$}} if there exists a number field $F$ satisfying the following two conditions:

\vspace{0.5mm}

{\rm{(a)}} $t_0$, $t_1$ $\in \mathbb{P}^1(F)$,

\vspace{0.5mm}

{\rm{(b)}} $t_0$ and $t_1$ meet modulo some prime ideal of the ring of integers of $F$ lying over $p \Zz$.
\end{defi}

\begin{rem} \label{3.2}
{\rm{(1)}} Definition \ref{rencontre}{\rm{(2)}} does not depend on the number field $F$ such that $t_0$ and $t_1$ $\in \mathbb{P}^1(F)$.

\vspace{0.5mm}

\noindent
{\rm{(2)}} If a given $t_0 \in \mathbb{P}^1(\Qq)$ meets a given $t_1 \in \mathbb{P}^1(\overline{\Qq})$ modulo a given prime number $p$, then $t_0$ meets each $\Qq$-conjugate of $t_1$ modulo $p$. 
\end{rem}

We shall need the following lemma:

\begin{lem} \label{lemma3}
Let $p$ be a prime number and $r$ the number of branch points of $E/\Qq(T)$. Suppose $p \geq r+1$. Then there exists $t_0 \in \Qq$ such that $t_0$ does not meet any branch point of $E/\Qq(T)$ modulo $p$.
\end{lem}

\begin{proof}
We claim that the reduction modulo~$p$ of any $t_0 \in \ZZ$ that meets one of the $r$ branch points lies in a subset of $\FF_p$ of cardinality at most~$r$. This implies the lemma because of the assumption $p > r$.

Let $t_1, \dots, t_{r'}$ be representatives of the branch points of $E/\Qq(T)$ for the action of $\GQ$. By Remark \ref{3.2}(2), it suffices to study the reductions modulo~$p$ of $t_0 \in \ZZ$ meeting $t_i$ for $1\le i \le r'$. Given $i \in \{1, \dots, r'\}$, denote the ring of integers of $\Qq(t_i)$ by $A$. By Remark \ref{3.2}(1), there exists a (non-zero) prime ideal $\mathfrak{P}$ of $A$ containing $p$ such that one of the following conditions holds:

\noindent
{\rm{(1)}} $v_\mathfrak{P}(t_0) \geq 0$, $v_\mathfrak{P}(t_i) \geq 0$ and $v_\mathfrak{P}(t_0-t_i) >0$,

\noindent
{\rm{(2)}} $v_\mathfrak{P}(t_0) \leq 0$, $v_\mathfrak{P}(t_i) \leq 0$ and $v_\mathfrak{P}((1/t_0)-(1/t_i)) >0$.

\noindent
First, assume (1) holds. Then the reduction $\overline{t_0} \in \mathbb{F}_p$ of $t_0$ modulo $p$  has to be equal to the reduction $\overline{t_i} \in A/\mathfrak{P}$ of $t_i$ modulo $\mathfrak{P}$. Now, assume (2) holds. If $v_\mathfrak{P}(t_i)=0$, then one has $v_\mathfrak{P}(t_0) =0$ as well (as $v_\mathfrak{P}((1/t_0)-(1/t_i)) >0$) and $v_\mathfrak{P}(t_0-t_i) = v_\mathfrak{P}((1/t_0)-(1/t_i)) >0$, i.e., (1) holds. One may then assume $v_\mathfrak{P}(t_i)<0$ and, consequently, $v_\mathfrak{P}(t_0)$ is negative as well, which cannot happen as $t_0 \in \Zz$.

This means that, for fixed $1\le i \le r'$, the reduction of $t_0$ modulo~$p$ lies in a subset of $\FF_p$ of cardinality at most the number of prime ideals lying over $p \Zz$, which is at most $[\Qq(t_i):\Qq]$. The claim follows because $\sum_{i=1}^{r'} [\Qq(t_i):\Qq] = r$.
\end{proof}

\begin{defi}
Let $p$ be a prime number. Say that $E/\Qq(T)$ has {\em{vertical ramification}} at $p$ if the prime ideal $p \Zz[T]$ of $\Zz[T]$ ramifies in the integral closure of $\Zz[T]$ in $E$.
\end{defi}

This practical test for non-vertical ramification is well-known (see, e.g., \cite[Addendum 1.4(c)]{DG12}):

\begin{prop} \label{nvr}
Let $p$ be a prime number. Suppose that there exists a monic separable polynomial $P(T,Y) \in \Zz[T][Y]$ that satisfies the following two conditions:

\vspace{0.5mm}

\noindent
{\rm{(1)}} the field $E$ is the splitting field over $\Qq(T)$ of $P(T,Y)$,

\vspace{0.5mm}

\noindent
{\rm{(2)}} the discriminant $\Delta(T) \in \Zz[T]$ of $P(T,Y)$ is not in $p\Zz[T]$.

\vspace{0.5mm}

\noindent
Then the extension $E/\Qq(T)$ has no vertical ramification at $p$.
\end{prop}

Finally, we recall the following result, which is part of the ``Specialisation Inertia Theorem" (see \cite[Proposition 4.2]{Bec91} and \cite[\S2.2.3]{Leg16}):

\begin{prop} \label{sit}
Let $p$ be a prime number and $t_0 \in \Qq$. The specialisation of $E/\Qq(T)$ at $t_0$ is unramified at $p$, provided the following two conditions hold:

\vspace{0.5mm}

\noindent
{\rm{(1)}} the extension $E/\Qq(T)$ has no vertical ramification at $p$,

\vspace{0.5mm}

\noindent
{\rm{(2)}} $t_0$ does not meet any branch point of $E/\Qq(T)$ modulo $p$.
\end{prop}

We shall also need the following result, which is a special case of \cite[Proposition 6.3]{KLN19}:

\begin{prop} \label{kln}
Let $E/\Qq(T)$ be a $\Qq$-regular Galois extension, let $t_1, \dots, t_r$ be the branch points of $E/\Qq(T)$ and let $F$ be the compositum of the residue fields $(E(t_1))_{t_1}, \dots, (E(t_r))_{t_r}$ of $E/\Qq(T)$ at $t_1,\dots,t_r$. Moreover, let $p$ be a prime number that is totally split in $F/\Qq$ (avoiding a finite set of prime numbers depending only on $E/\Qq(T)$). Then the decomposition group of $E_{t_0}/\Qq$ at $p$ is cyclic for every $t_0 \in \Qq$.
\end{prop}

In the sequel, we shall also deal with the local behaviour of specialisations of $E/\Qq(T)$ at the infinite prime. In this context, the following proposition is useful:

\begin{prop} \label{df}
Denote the branch points of $E/\Qq(T)$ by $t_1, \dots, t_r$ and let $t_0 \in \Qq \setminus \{t_1, \dots, t_r\}$.

\vspace{0.5mm}

\noindent
{\rm{(1)}} Suppose $E/\Qq(T)$ has three branch points and ${\rm{Gal}}(E/\Qq(T))$ is not dihedral of order $4$, $6$, $8$, $12$. Then $E_{t_0}/\Qq$ is not totally real.

\vspace{0.5mm}

\noindent
{\rm{(2)}} Suppose there exists a monic separable polynomial $P(T,Y) \in \Qq[T][Y]$ of splitting field $E$ over $\Qq(T)$ and an integer $0 \leq n \leq {\rm{deg}}_Y P-2$ such that $P(t_0,Y)$ is separable and the $n$-th derivative of $P(t_0,Y)$ has at least one complex non-real root. Then $E_{t_0}/\Qq$ is not totally real.
\end{prop}

\begin{proof}
(1) This is \cite[Proposition 1.2]{DF90}.

\vspace{0.5mm}

\noindent
(2) We reproduce the proof of \cite[Lemma 2.3]{LSY12} in a more general context. Suppose $E_{t_0}/\Qq$ is totally real. Then all roots of $P(t_0,Y)$ are real. As this polynomial is also separable, we obtain, by Rolle's theorem, that the derivative $P'(t_0,Y)$ has at least ${\rm{deg}}_Y P(t_0,Y) -1$ distinct real roots, that is, $P'(t_0,Y)$ is separable and has only real roots. It then suffices to iterate this argument to get a contradiction.
\end{proof}

The following well-known result shows that, to construct specialisations of $E/\Qq(T)$ with full Galois group and with specified local behaviour at finitely many given rational places (possibly infinite), one can look at one prime at a time and we do not have to worry about the corresponding Galois group:

\begin{prop} \label{pv}
Let $\mathcal{S}$ be a finite set of rational places. For each $p \in \mathcal{S}$, fix a Galois extension $F_p/\Qq_p$ \footnote{Set $\Qq_\infty=\mathbb{R}$ if $p =\infty$.} whose Galois group embeds into $G = {\rm{Gal}}(E/\Qq(T))$. Suppose that, for each $p \in \mathcal{S}$, there exists $t_{0,p} \in \Qq$, not a branch point of $E/\mathbb{Q}(T)$, such that $F_p$ is the completion of $E_{t_0,p}$ at $p$. Then there exists $t_0 \in \Qq$ such that ${\rm{Gal}}(E_{t_0}/\Qq)=G$ and, for each $p \in \mathcal{S}$, the field $F_p$ is the completion of $E_{t_0}$ at $p$. Moreover, the set of all extensions $E_{t_0}/\Qq$ with these properties is infinite.
\end{prop}

\begin{proof}
The existence of at least one specialisation $E_{t_0}/\Qq$ with the above properties can be found in, e.g., \cite[Proposition 2.1]{PV05}. To conclude that there exist infinitely many distinct such extensions $E_{t_0}/\Qq$, it suffices to iterate the above statement, combined with the fact that the set of all prime numbers $p$ for which there exists $t_0 \in \Qq$ such that $p$ ramifies in $E_{t_0}/\Qq$ is infinite (see \cite[Corollary 2.12]{Leg16}).
\end{proof}

\section{Galois families} \label{sec:gal_fam}

A purpose of this text is to propose the following definition, whose part (2) is Definition A from the introduction:

\begin{defi}\label{defi:GF}
Let $I$ be a set (of indices) and, for each $i \in I$, let $f_i$ be a normalised Hecke eigenform (of any level and weight).

\vspace{1mm}

\noindent
{\rm{(1)}} For a prime number~$p$, a positive integer~$n$ and a finite subgroup $G \subset \GL_2(\overline{\mathbb{F}_p})$, we say that the $(f_i)_{i\in I}$ form an {\em $n$-parameter $G$-Galois family} if there exists a finite Galois extension $E/\QQ({\bf{T}})=E/\Qq(T_1, \dots, T_n)$ with Galois group isomorphic to~$G$ such that, for each $i \in I$, the following two conditions hold:

\vspace{0.5mm}

{\rm{(a)}} there is ${{\bf t}}_i \in \QQ^n$ such that $K_{f_i}=E_{{\bf t}_i}$ and

\vspace{0.5mm}

{\rm{(b)}} the image $\Image(\rho_{f_i})$ is conjugate to~$G$ in $\GL_2(\overline{\Ff_p})$ \footnote{Recall that $\Image(\rho_{f_i})$ is uniquely determined up to conjugation by $f_i$.}.

\vspace{1mm}

\noindent
{\rm{(2)}} For a finite subgroup $G \subset \PGL_2(\overline{\mathbb{F}_p})$, we define an {\em $n$-parameter projective $G$-Galois family} exactly as above, with the only exception that we replace $K_{f_i}$ by $K_{f_i}^\proj$ (for $i \in I$).

\vspace{1mm}

\noindent
{\rm{(3)}} When $G \subset \GL_2(\CC)$ (resp., $G \subset \PGL_2(\CC)$) and the $f_i$ for $i \in I$ are holomorphic weight one forms, we make exactly the same definition via the attached Artin representations and call this family an {\em $n$-parameter Artin $G$-Galois family} (resp., an {\em $n$-parameter projective Artin $G$-Galois family}).

\vspace{1mm}

\noindent
{\rm{(4)}} An $n$-parameter (projective) (Artin) $G$-Galois family is called {\em regular} if the underlying extension $E/\Qq({\bf{T}})$ is $\Qq$-regular.
\end{defi}

We remark that the base field $\QQ$ could be replaced by any number field in the definition if both automorphic and field extension sides are changed accordingly. Moreover, we are not explicitly insisting that our Galois families are infinite; any set $I$ is allowed. We are primarily interested in families where infinitely many pairwise non-isomorphic fields occur as $K_{f_i}$ or $K_{f_i}^\proj$.

Note that, by (b) in the definition of Galois families, the images $\Image(\rho_{f_i})$ are all conjugate to the fixed subgroup $G$ of the general linear group. So, by choosing appropriate bases for the representation modules underlying $\rho_{f_i}$ for $i \in I$, we could actually assume that they are {\em equal}.

The field extension $E/\QQ({\bf{T}})$ underlying a Galois family $(f_i)_{i \in I}$ can also be viewed via a Galois representation
$$ \rho: \Gal(\overline{\QQ({\bf{T}})}/\QQ({\bf{T}})) \twoheadrightarrow \Gal(E/\QQ({\bf{T}})) \cong G \subset \GL_2(\overline{\Ff_p}).$$
The representations $\rho_{f_i}$ can then be interpreted as specialisations of $\rho$. Furthermore, letting $H$ be the kernel of $G \to \GL_2(\overline{\Ff_p}) \twoheadrightarrow \PGL_2(\overline{\Ff_p})$ (its image equals $G^\proj$), we have the associated projective Galois representation
$$\rho^\proj: \Gal(\overline{\QQ({\bf{T}})}/\QQ({\bf{T}})) \twoheadrightarrow \Gal(E^H/\QQ({\bf{T}})) \cong G^\proj \subset \PGL_2(\overline{\Ff_p}).$$
Similar statements are true in the Artin case.

Viewing the natural isomorphism between $\Gal(E_{{\bf{t}}_i}/\QQ)$ and $\Gal(E/\QQ({\bf{T}}))$ as equality and considering $H$ as a subgroup of both, we have the equality
$$K_{f_i}^\proj = (K_{f_i})^H = (E_{{\bf{t}}_i})^H = (E^H)_{{\bf{t}}_i}$$
of number fields. It shows that all $K_{f_i}^\proj$ are obtained as specialisations of the extension $E^H/\QQ({\bf{T}})$, and as $\Image(\rho_{f_i})$ is conjugate to $G$, the image $\Image(\rho_{f_i}^\proj)$ is conjugate to~$G^\proj$. This proves this result:

\begin{prop}\label{prop:getreg}
Let $(f_i)_{i \in I}$ be a (regular) $n$-parameter (Artin) $G$-Galois family. Then $(f_i)_{i \in I}$ is a (regular) $n$-parameter projective (Artin) $G^\proj$-Galois family.
\end{prop}

The direct converse of the proposition is not true because a given finite subgroup of $\PGL_2(\overline{\Ff_p})$ comes from infinitely many different finite subgroups of $\GL_2(\overline{\Ff_p})$.

\section{The Beckmann-Black Problem, parametric extensions and Galois fa\-mi\-lies} \label{sec:BB}

\subsection{The Beckmann-Black Problem and Galois fa\-mi\-lies}

First, we recall the Beckmann-Black Problem (over $\Qq$), which was intensively studied (see, e.g., the survey paper \cite{Deb01b} for more details and references).

\begin{problem}[Beckmann-Black Problem]\label{prob:BB}
Let $G$ be a finite group. Is it true that every Galois extension $F/\Qq$ with Galois group $G$ occurs as a specialisation of some $\Qq$-regular Galois extension $E/\Qq(T)$ with Galois group $G$ (possibly depending on $F/\Qq$)?
\end{problem}

Let us recall that we take $\cG \in \{\GL_2(\overline{\Ff_p}), \GL_2(\CC)\}$ and $\cP \in \{\PGL_2(\overline{\Ff_p}), \PGL_2(\CC)\}$.
We also remind the reader of our convention that if a representation $\rho: \GQ \to \cG$ (resp., a projective represention $\rho^\proj: \GQ \to \cP$) comes from a normalised Hecke eigenform~$f$, then we assume $f$ to be holomorphic of weight one if we are in the Artin case.

Translating the Beckmann-Black Problem to the language of modular forms leads us to propose the following new problem:

\begin{problem}\label{prob:mf}
Let $G \subset \cG$ (resp., $G \subset \cP$) be a finite subgroup. Does every normalised Hecke eigenform $f$ such that $\Image(\rho_f)$ (resp., $\Image(\rho_f^\proj)$) is conjugate to $G$ belong to some regular $1$-parameter (Artin) $G$-Galois family (resp., some regular $1$-parameter projective (Artin) $G$-Galois family), possibly depending on $f$?
\end{problem}

Note that Proposition \ref{prop:getreg} implies that a positive answer for a given finite subgroup $G \subset \cG$ automatically gives a positive answer for the image $G^\proj$ of $G$ under the natural map $\cG \twoheadrightarrow \cP$. The following proposition makes the gap between Problems \ref{prob:BB} and \ref{prob:mf} precise:

\begin{prop}\label{rem:bbpb1}
Let $G \subset \cG$ (resp., $G \subset \cP$) be a finite irreducible subgroup. The answer to Problem \ref{prob:mf} is affirmative if and only if 

\vspace{0.5mm}

\noindent
- every Galois extension $F$ of $\Qq$ of group $G$ and such that $F^{\text{scalars in }G}$ (resp., $F$) is totally imaginary occurs as a specialisation of a $\Qq$-regular Galois extension of $\Qq(T)$ of group $G$ if we are in the Artin case or in the mod $p$ case with $p \geq 3$,

\vspace{0.5mm}

\noindent
- Problem \ref{prob:BB} has an affirmative answer if we are in the mod $p$ case with $p=2$.
\end{prop}

\begin{proof}
We prove only the general linear case over $\overline{\Ff_p}$ as the proofs in the Artin case and the projective cases are almost identical. First, suppose the answer to Problem \ref{prob:mf} is affirmative. Let $F/\Qq$ be a Galois extension of group $G$ such that $F^{\text{scalars in }G}$ is totally imaginary if $p$ is odd. By Proposition \ref{tool0}, there exists a normalised Hecke eigenform $f$ such that $K_f=F$. Then, from our assumption, there exists a $\Qq$-regular Galois extension of $\Qq(T)$ of group $G$ which specialises to $K_f / \Qq$ at some $t_0 \in \Qq$. As $K_f=F$, we are done. Now, assume every Galois extension $F$ of $\Qq$ of group $G$ such that $F^{\text{scalars in }G}$ is totally imaginary if $p$ is odd occurs as a specialisation of a $\Qq$-regular Galois extension of $\Qq(T)$ of group $G$. Let $f$ be a normalised Hecke eigenform such that $\rho_{f,p}$ has image $G$. If $p$ is odd, by Proposition \ref{tool0}, the field $K_f^{\proj}$ is totally imaginary. Then, from our assumption, there exists a $\Qq$-regular Galois extension of $\Qq(T)$ of group $G$ that specialises to $K_f /\Qq$ at some $t_0 \in \Qq$, thus leading to the desired conclusion.
\end{proof}

If a given Galois extension $F/\Qq$ of group $G$ occurs as a specialisation of some $\Qq$-regular Galois extension $E/\Qq(T)$ of group $G$, Hilbert's irreducibility theorem shows that $F/\Qq$ belongs to an infinite family of specialisations of $E/\Qq(T)$ of group $G$. Below we show that the same conclusion holds in the context of modular forms:

\begin{prop}\label{prop:bbpb2}
Let $f$ be a normalised cuspidal Hecke eigenform. Suppose $\rho_{f}$ (resp., $\rho_f^\proj$) is irreducible with image conjugate to $G \subset \cG$ (resp., $G \subset \cP$). If $f$ belongs to a regular 1-parameter (Artin) $G$-Galois family (resp., a regular 1-parameter projective (Artin) $G$-Galois family), then $f$ belongs to an infinite regular $1$-parameter (Artin) $G$-Galois family (resp., an infinite regular 1-parameter projective (Artin) $G$-Galois family).
\end{prop}

\begin{proof}
We prove only the general linear case over $\overline{\Ff_p}$ as the proofs in the Artin case and the projective cases are almost identical. The extension $K_f/\QQ$ is Galois with group $G = \Image(\rho_f) \subset \GL_2(\overline{\Ff_p})$. By assumption, there exists a $\Qq$-regular Galois extension $E/\QQ(T)$ of group $G$ giving rise to $K_f/\Qq$ by specialisation at some $t_0 \in \QQ$. By Hilbert's irreducibility theorem, infinitely many distinct Galois extensions of $\Qq$ of group $G$ occur as specialisations of $E/\Qq(T)$. Hence, by Proposition \ref{tool0}, we get the desired conclusion if $p=2$. If $p$ is odd, then $K_f^{\proj}$ is totally imaginary by Proposition \ref{tool0}. Proposition~\ref{pv} applied simultaneously to $E/\mathbb{Q}(T)$ and $E^{\text{scalars in }G}/\mathbb{Q}(T)$ then provides infinitely many distinct  Galois extensions of $\Qq$ of group $G$ such that their subfields fixed by the scalars in $G$ are totally imaginary occuring as specialisations of $E/\Qq(T)$. As in the case $p=2$, we apply Proposition \ref{tool0} to conclude.
\end{proof}

\subsection{Parametric extensions and Galois families}

Let us now state the following definition, which is a function field analogue of the classical notion of 'parametric polynomial' as defined in \cite[Definition 0.1.1]{JLY02} (recalled as Definition \ref{def:par_gen2}):

\begin{defi} \label{def:par_gen}
Let ${\bf{T}}=(T_1, \dots, T_n)$ be an $n$-tuple of algebraically independent indeterminates ($n \geq 1$) and $E/\Qq({\bf{T}})$ a finite Galois extension of group $G$. Say that $E/\Qq({\bf{T}})$ is {\em{parametric}} if every Galois extension $F/\Qq$ of group $G$ occurs as the specialisation $E_{\bf{t}}/\Qq$ of $E/\Qq({\bf{T}})$ at some ${\bf{t}} \in \Qq^n$.
\end{defi}

Translating the notion of parametric extension to the language of modular forms leads us to propose the following new definition:

\begin{defi} \label{def:para_mf}
Let $n$ be a positive integer and $G \subset \cG$ (resp., $G \subset \cP$) a finite subgroup. An $n$-parameter (Artin) $G$-Galois family (resp., an $n$-parameter projective (Artin) $G$-Galois family) $(f_i)_{i \in I}$ is called {\em parametric} if, for any normalised Hecke eigenform $f$ such that the image $\Image (\rho_{f})$ is conjugate to~$G$ in $\cG$ (resp., the image $\Image (\rho_{f}^\proj)$ is conjugate to~$G$ in $\cP$), there is $i \in I$ such that $K_{f_i} = K_f$ (resp., $K_{f_i}^\proj = K_f^\proj$).
\end{defi}

The following proposition is the analogue of Proposition \ref{rem:bbpb1} in the parametric context:

\begin{prop} \label{pp->pp_mf}
Let $n$ be a positive integer and $G \subset \cG$ (resp., $G \subset \cP$) an irreducible finite subgroup. Then there is a parametric $n$-parameter (Artin) $G$-Galois family (resp., a parametric $n$-parameter projective (Artin) $G$-Galois family) if and only if 

\vspace{0.5mm}

\noindent
- there is a Galois extension $E/\Qq(T_1,\dots, T_n)=E/\Qq({\bf{T}})$ of group $G$ such that every Galois extension $F$ of $\Qq$ of group $G$ satisfying that $F^{\text{scalars in }G}$ (resp., $F$) is totally imaginary occurs as a specialisation of $E/\Qq({\bf{T}})$ if we are in the Artin case or in the mod $p$ case with $p \geq 3$,

\vspace{0.5mm}

\noindent
- there is a Galois extension $E/\Qq(T_1,\dots, T_n)=E/\Qq({\bf{T}})$ of group $G$ that is parametric if we are in the mod $p$ case with $p=2$.

\vspace{0.5mm}

\noindent
Moreover, the (projective) (Artin) $G$-Galois family is regular if and only if  there is $E/\Qq({\bf{T}})$ as above which, in addition, is $\Qq$-regular.
\end{prop}

\begin{proof}
We prove only the general linear case for $p \geq 3$ as the proofs in all other cases are almost identical. For an arbitrary $n$-parameter $G$-Galois family $(f_i)_{i \in I}$ with underlying function field extension $E/\Qq({\bf{T}})$, Proposition \ref{tool0} provides
$$\mathcal{S}_1:=\{K_{f_i}/\Qq \, : \, i \in I\} \subseteq 
  \mathcal{S}_2:= \{E_{\bf{t}}/\Qq \, : \, {\bf{t}} \in \Qq^n, \, {\rm{Gal}}(E_{\bf{t}}/\Qq)=G, \, {\rm{and}} \,  E_{\bf{t}}^{\text{scalars in }G} \, {\rm{totally}} \, {\rm{imaginary}}\}$$
$$\hspace{2cm} \subseteq \mathcal{S}_3:=\{F/\Qq \, : \, {\rm{Gal}}(F/\Qq)=G \, {\rm{and}} \,  F^{\text{scalars in }G} \, {\rm{totally}} \, {\rm{imaginary}}\}.$$
Moreover, $\mathcal{S}_3$ is equal to 
$$\mathcal{S}_4:=\{K_f/\Qq \, : \, f \, {\rm{normalised}} \,  {\rm{Hecke}} \,  {\rm{eigenform}} \,  {\rm{with}} \,  {\rm{im}}(\rho_f) \, {\rm{conjugate}} \, {\rm{to}} \, G \, {\rm{in}} \, {\rm{GL}}_2(\overline{\Ff_p})\}.$$
In particular, if $(f_i)_{i \in I}$ is parametric, then $\mathcal{S}_1 = \mathcal{S}_4$. Consequently, one has $\mathcal{S}_2 = \mathcal{S}_3$, as needed.

Conversely, suppose there exists a Galois extension $E/\Qq(T_1, \dots, T_n)$ of group $G$ such that every Galois extension $F$ of $\Qq$ of group $G$ satisfying that $F^{\text{scalars in }G}$  is totally imaginary occurs as a specialisation of $E/\Qq(T_1, \dots, T_n)$. Let $(f_i)_{i \in I}$ be the family of all normalised Hecke eigenforms such that $\Image(\rho_f)$ is conjugate to $G$ in ${\rm{GL}}_2(\overline{\Ff_p})$. By Proposition \ref{tool0} and our assumption, $(f_i)_{i \in I}$ is an $n$-parameter $G$-Galois family, which is trivially parametric.
\end{proof}

Given a finite group $G$, it is well-known that, if there exists a $\Qq$-parametric polynomial $P({\bf{T}},Y) \in \Qq[{\bf{T}}][Y]$ of group $G$ such that $E/\Qq({\bf{T}})$ is $\Qq$-regular, where $E$ is the splitting field over $\Qq({\bf{T}})$ of $P({\bf{T}},Y)$, then the Beckmann-Black Problem has a positive answer for the group $G$ (see, e.g., \cite[Proposition 3.3.10]{JLY02}). Below we show that the same conclusion holds in the context of modular forms:

\begin{prop} \label{prop:going_up}
Let $G \subset \cG$ or $G \subset \cP$ be a finite subgroup. If there exists a regular parametric $n$-parameter (projective) $G$-Galois family (for some $n \geq 1$), then the answer to Problem \ref{prob:mf} is affirmative.
\end{prop}

\begin{proof}
We prove only the general linear case over $\overline{\Ff_p}$ as the proofs of the other cases are almost identical. Let $E/\Qq({\bf{T}}) = E/\Qq(T_1, \dots, T_n)$ be the $\Qq$-regular Galois extension of group $G$ underlying the regular parametric $n$-parameter $G$-Galois family from the statement and let $f$ be a normalised Hecke eigenform such that the image of $\rho_{f,p}$ is conjugate to $G$ in ${\rm{GL}}_2(\overline{\Ff_p})$. Pick ${\bf{\alpha}}=(\alpha_1, \dots, \alpha_n) \in \Qq^n$ such that the number field $K_f$ is the specialised field $E_{\bf{\alpha}}$. We also fix ${\bf{\beta}}=(\beta_1, \dots, \beta_n) \in \Qq^n$ such that $E_\beta/\Qq$ has Galois group $G$ and such that the fields $E_\alpha$ and $E_\beta$ are linearly disjoint over $\Qq$; such a $\beta$ exists as $E/\Qq({\bf{T}})$ is $\Qq$-regular. Now, given a new indeterminate $T$, for each $i \in \{1, \dots, n\}$, we fix $a_i(T) \in \Qq[T]$ such that $a_i(0)=\alpha_i$ and $a_i(1)=\beta_i$. We set ${\bf{a}} = (a_1(T), \dots, a_n(T))$. Consider the $\Qq(T)$-regular Galois extension $E(T)/\Qq(T)(T_1, \dots, T_n)$ of group $G$ and its specialisation $(E(T))_{\bf{a}} / \Qq(T)$ at ${\bf{a}}$. Below we show that the specialisation of $(E(T))_{\bf{a}} / \Qq(T)$ at 0 (resp., at 1) is the extension $E_\alpha/\Qq$ (resp., $E_\beta/\Qq$). Consequently, the extension $(E(T))_{\bf{a}} / \Qq(T)$ has Galois group $G$ (since this holds for $E_\alpha/\Qq$) and, as $E_\alpha \cap E_\beta = \Qq$, the extension $(E(T))_{\bf{a}} / \Qq(T)$ is $\Qq$-regular.

Let $B$ be the integral closure of $\Qq[{\bf{T}}]$ in $E$, $\mathfrak{P}$ a maximal ideal of the integral closure of $\Qq(T)[{\bf{T}}]$ in $E(T)$ lying over $\langle T_1 - a_1(T), \dots, T_n - a_n(T) \rangle$, $\widetilde{B}$ the integral closure of $\Qq[T]$ in $(E(T))_{\bf{a}}$ and $\mathfrak{P}_0$ a maximal ideal of $\widetilde{B}$ lying over $\langle T \rangle$. Since the reduction modulo $\mathfrak{P}$ of an element of $B$ yields an element of $\widetilde{B}$, we get a well-defined homomorphism
$$\psi : B \rightarrow \widetilde{B} / \mathfrak{P}_0.$$
Moreover, since $a_i(0)=\alpha_i$ for every $i \in \{1, \dots, n\}$, one has 
$$\langle T_1 - \alpha_1, \dots, T_n - \alpha_n \rangle \subseteq {\rm{ker}}(\psi) \cap \Qq[{\bf{T}}],$$
that is, 
$$\langle T_1 - \alpha_1, \dots, T_n - \alpha_n \rangle = {\rm{ker}}(\psi) \cap \Qq[{\bf{T}}],$$
as the ideal in the left-hand side is maximal and $\Qq[{\bf{T}}] \not \subseteq {\rm{ker}}(\psi)$. Consequently, the ideal ${\rm{ker}}(\psi)$ of $B$ lies over $\langle T_1 - \alpha_1, \dots, T_n - \alpha_n \rangle$ and it is maximal. One then has
$$E_\alpha = B/{\rm{ker}}(\psi) \subseteq \widetilde{B}/\mathfrak{P}_0 = ((E(T))_{\bf{a}})_0.$$
As the field in the left-hand side has degree $|G|$ over $\Qq$ and that in the right-hand side has degree at most $|G|$ over $\Qq$, we get the desired equality $E_\alpha = ((E(T))_{\bf{a}})_0.$ Similarly, one has $E_\beta = ((E(T))_{\bf{a}})_1$.
\end{proof}

\subsection{Explicit examples}

We conclude this section by giving explicit examples of finite groups $G$ for which the answer to Problem \ref{prob:mf} is affirmative and/or there exists a parametric (projective) (Artin) $G$-Galois family. To that end, we use the previous results from this section, thus meaning that the groups we choose below are known to have a generic polynomial over $\Qq$ and/or to fulfill the Beckmann--Black Problem. Of course, there are more groups fulfilling this condition than just those given in the next theorem and we invite the interested reader to give more examples. 

We start with the mod $p$ case.

\begin{thm}\label{thm:parametric_poly}
Let $p$ be a prime number.

\vspace{0.5mm}

\noindent
{\rm{(1)}} Let $G \subset {\rm{PGL}}_2(\overline{\Ff_p})$ be a subgroup isomorphic to any of the following finite groups: $\Zz/2\Zz \times \Zz/2\Zz$, the dihedral group $D_4$ with eight elements, $A_4$, $S_4$, $A_5$, $S_5$. Assume 

\vspace{0.5mm}

- $p \geq 3$ if $G=\Zz/2\Zz \times \Zz/2\Zz$ or $D_4$, 

\vspace{0.5mm}

- $p \geq 5$ if $G=A_4$ or $S_4$, 

\vspace{0.5mm}

- $p \geq 7$ if $G=A_5$,

\vspace{0.5mm}

- $p=5$ if $G=S_5$. 

\vspace{0.5mm}

\noindent
Then the following two conclusions hold.

\vspace{0.5mm}

{\rm{(a)}} There is a regular parametric 2-parameter projective $G$-Galois family.

\vspace{0.5mm}

{\rm{(b)}} For every normalised Hecke eigenform $f$ such that the Galois group ${\rm{Gal}}(K_f^\proj/\Qq)$ is conjugate 

to $G$, there exists an infinite regular 1-parameter projective $G$-Galois family containing $f$. In parti-

cular, the answer to Problem \ref{prob:mf} is affirmative.

\vspace{0.5mm}

\noindent
{\rm{(2)}} Let $G \subset {\rm{PGL}}_2(\overline{\Ff_p})$ be a subgroup isomorphic to any of the following finite groups: the dihedral group $D_m$ with $2m$ elements ($m$ odd) or the dihedral group $D_8$ with 16 elements. Assume

\vspace{0.5mm}

- $p$ is odd (in both cases),

\vspace{0.5mm}

- $p$ does not divide $m$ (in the former case). 

\vspace{0.5mm}

\noindent
Then the following two conclusions hold.

\vspace{0.5mm}

{\rm{(a)}} There is a regular parametric n-parameter projective $G$-Galois family for some $n \geq 1$.

\vspace{0.5mm}

{\rm{(b)}} For every normalised Hecke eigenform $f$ such that the Galois group ${\rm{Gal}}(K_f^\proj/\Qq)$ is conjugate 

to $G$, there exists an infinite regular 1-parameter projective $G$-Galois family containing $f$. In parti-

cular, the answer to Problem \ref{prob:mf} is affirmative.

\vspace{0.5mm}

\noindent
{\rm{(3)}} Assume $p=3$. Let $G$ be the finite group $A_6 \cong {\rm{PSL}}_2(\Ff_9)$ and let $f$ be a normalised Hecke eigenform such that the Galois group ${\rm{Gal}}(K_f^\proj/\Qq)$ is conjugate to $G$. Then there is an infinite regular 1-parameter projective $G$-Galois family containing $f$. In particular, the answer to Problem \ref{prob:mf} is affirmative.
\end{thm}

We shall need the following definition:

\begin{defi} \label{def:par_gen2}
Given a positive integer $n$, let ${\bf{T}}=(T_1, \dots, T_n)$ be an $n$-tuple of algebraically independent indeterminates and $P({\bf{T}},Y) \in \Qq[{\bf{T}}][Y]$ a monic separable polynomial. Denote the Galois group of $P({\bf{T}},Y)$ over $\Qq({\bf{T}})$ by $G$.

\vspace{0.5mm}

\noindent
{\rm{(1)}} Let $k$ be a field containing $\Qq$. Say that $P({\bf{T}},Y)$ is {\em{$k$-parametric}} if, for every Galois extension $F/k$ of group $G$, the field $F$ is the splitting field over $k$ of some polynomial $P({\bf{t}},Y)$ with ${\bf{t}} \in k^n$.

\vspace{0.5mm}

\noindent
{\rm{(2)}} Say that $P({\bf{T}}, Y)$ is {\em{generic}} if it is $k$-parametric for every field $k$ containing $\Qq$.
\end{defi}

\begin{rem} \label{!}
Let $G$ be a finite group and $P({\bf{T}},Y) \in \Qq[{\bf{T}}][Y]$ a monic separable polynomial of group $G$ and splitting field $E$ over ${\bf{\Qq}}({\bf{T}})$. 

\vspace{0.5mm}

\noindent
{\rm{(1)}} If $P({\bf{T}},Y)$ is $\Qq$-parametric, then $E/\Qq({\bf{T}})$ is parametric (see \S\ref{ssec:ffe}).

\vspace{0.5mm}

\noindent
{\rm{(2)}} If $P({\bf{T}},Y)$ is generic, then $E/\Qq({\bf{T}})$ is $\Qq$-regular (see \cite[Proposition 3.3.8]{JLY02}).
\end{rem}

\begin{proof}[Proof of Theorem \ref{thm:parametric_poly}]
(1) By \cite[page 203]{JLY02}, the group $G$ has a generic polynomial $P(T_1, T_2,Y) \in \Qq[T_1,T_2][Y]$. Consequently, the fact that (a) holds is a consequence of Proposition \ref{pp->pp_mf} and Remark \ref{!}. As for (b), it is a consequence of (a), Proposition \ref{prop:bbpb2} and Proposition \ref{prop:going_up} (note that irreduciblity is guaranteed as it is easy to see that $G$ is not isomorphic to any quotient of a finite subgroup of the upper triangular matrices inside $\GL_2(\overline{\Ff_p})$).

\noindent
(2) The proof is identical to the proof of (1). The group $G$ has a generic polynomial with rational coefficients (see, e.g., \cite[page 112]{JLY02}).

\noindent
(3) Here we use that the Beckmann-Black Problem has a positive answer for the group $G$ (see, e.g., \cite[th\'eor\`eme 2.2]{Deb01b}) and apply Propositions \ref{rem:bbpb1} and \ref{prop:bbpb2}.
\end{proof}

Now, we give the analogue of Theorem \ref{thm:parametric_poly} in the Artin situation. As the proof is almost identical to the previous one, details are left to the reader.

\begin{thm} \label{thm:parametric_poly2}
{\rm{(1)}} Let $G\subset {\rm{PGL}}_2(\Cc)$ be a subgroup isomorphic to any of the following finite groups: $\Zz/2\Zz \times \Zz/2\Zz$, the dihedral group $D_4$ with eight elements, $A_4$, $S_4$, $A_5$. Then these conclusions hold.

\vspace{0.5mm}

{\rm{(a)}} There is a regular parametric 2-parameter projective Artin $G$-Galois family.

\vspace{0.5mm}

{\rm{(b)}} For every holomorphic normalised Hecke eigenform $f$ of weight one such that the Galois group 

${\rm{Gal}}(K_f^\proj/\Qq)$ is conjugate to $G$, there is an infinite regular 1-parameter projective Artin $G$-Galois 

family containing $f$. In particular, the answer to Problem \ref{prob:mf} is affirmative.

\vspace{0.5mm}

\noindent
{\rm{(2)}} Let $G\subset {\rm{PGL}}_2(\Cc)$ be a subgroup isomorphic to any of the following finite groups: the dihedral group $D_m$ with $2m$ elements ($m$ odd) or the dihedral group $D_8$ with 16 elements.

\vspace{0.5mm}

{\rm{(a)}} There is a regular parametric n-parameter projective Artin $G$-Galois family for some $n \geq 1$.

\vspace{0.5mm}

{\rm{(b)}} For every holomorphic normalised Hecke eigenform $f$ of weight one such that the Galois group 

${\rm{Gal}}(K_f^\proj/\Qq)$ is conjugate to $G$, there is an infinite regular 1-parameter projective Artin $G$-Galois

family containing $f$. In particular, the answer to Problem \ref{prob:mf} is affirmative.
\end{thm}

Finally, we show that parametric 1-parameter projective (Artin) $G$-Galois families do not occur for several finite groups $G$:

\begin{thm} \label{?}
{\rm{(1)}} Let $p$ be a prime number and $G$ a finite irreducible subgroup of ${\rm{PGL}}_2(\overline{\Ff_p})$. Suppose the following three conditions hold:

\vspace{0.5mm}

{\rm{(a)}} $G$ has even order,

\vspace{0.5mm}

{\rm{(b)}} $G$ has a generic polynomial with rational coefficients,

\vspace{0.5mm}

{\rm{(c)}} $G$ has a non-cyclic abelian subgroup.

\vspace{0.5mm}

\noindent
Then there does not exist any parametric 1-parameter projective $G$-Galois family.

\vspace{0.5mm}

\noindent
{\rm{(2)}} Let $G$ be a finite irreducible subgroup of ${\rm{PGL}}_2(\Cc)$. Suppose the following three conditions hold:

\vspace{0.5mm}

{\rm{(a)}} $G$ has even order,

\vspace{0.5mm}

{\rm{(b)}} $G$ has a generic polynomial with rational coefficients,

\vspace{0.5mm}

{\rm{(c)}} $G$ has a non-cyclic abelian subgroup.

\vspace{0.5mm}

\noindent
Then there does not exist any parametric 1-parameter projective Artin $G$-Galois family.
\end{thm}

In particular, if $G$ is any finite group and $p$ any prime number as in Theorem \ref{thm:parametric_poly}(1), then 2 is the least integer $n$ such that there exists a (regular) parametric $n$-parameter projective $G$-Galois family\footnote{All these finite groups admit $\Zz/2\Zz \times \Zz/2\Zz$ as a subgroup.}. The same conclusion holds in the Artin situation for finite groups $G$ in Theorem \ref{thm:parametric_poly2}(1).

We shall need the following lemma:

\begin{lem} \label{kln2}
Let $G$ be a finite group, $m$ a positive integer, and $F_1/\Qq, \dots, F_m/\Qq$ finite Galois extensions of $\Qq$ whose Galois groups are subgroups of $G$. Suppose there exists a generic polynomial with rational coefficients and Galois group $G$. Then there exists a $\Qq$-regular Galois extension of $\Qq(T)$ of group $G$ which specialises to $F_1/\Qq, \dots, F_m/\Qq$ at non branch points.
\end{lem}

\begin{proof}
Since there exists a generic polynomial of group $G$ with rational coefficients, one may apply \cite{DeM83} and \cite[Theorem 5.2.5]{JLY02} to get that there exist an integer $n \geq 1$ and a polynomial $P({\bf{T}},Y)=P(T_1, \dots, T_n,Y) \in \Qq[{\bf{T}}][Y]$ of group $G$ such that, for every extension $L/\Qq$ and every Galois extension $F/L$ of group $H$ contained in $G$, there exists ${\bf{t}} \in L^n$ such that $P({\bf{t}},Y)$ is separable and $F$ is the splitting field over $L$ of $P({\bf{t}}, Y)$.

Pick a finite Galois extension $F_{m+1}/\Qq$ of group $G$ and set $F_{m+2}/\Qq=\Qq/\Qq$. By the above, for $i \in \{1, \dots, m+2\}$, there exists ${\bf{t}}_i \in \Qq^n$ such that $P({\bf{t}}_i,Y)$ is separable and the splitting field over $\Qq$ of $P({\bf{t}}_i,Y)$ is $F_i$. By polynomial interpolation (as in the proof of \cite[Proposition 3.3.10]{JLY02}), one constructs a monic polynomial $Q(T,Y) \in \Qq[T][Y]$ such that, for each $i \in \{1, \dots, m+2\}$, $Q(i,Y) = P({\bf{t}}_i,Y)$. Fix $i \in \{1, \dots, m+2\}$. Since $P({\bf{t}}_i,Y)$ is separable, $Q(T,Y)$ is also separable. Let $E$ be the splitting field of $Q(T,Y)$ over $\Qq(T)$. Since $Q(i,Y)$ is separable, the specialisation of $E/\Qq(T)$ at $i$ is $F_i/\Qq$ and $i$ is not a branch point of $E/\mathbb{Q}(T)$. It remains to notice that $E/\Qq(T)$ must be $\Qq$-regular (by using $F_{m+2}/\Qq)$ and has Galois group $G$ (by using $F_{m+1}/\Qq$) to conclude the proof.
\end{proof}

\begin{proof}[Proof of Theorem \ref{?}]
(1) Suppose there exists a parametric 1-parameter projective $G$-Galois family and denote the underlying function field extension by $E/\Qq(T)$. 

First, assume $E/\Qq(T)$ is not $\Qq$-regular. Then there exists a non-trivial finite Galois extension $L/\Qq$ such that, for every normalised Hecke eigenform $f$ such that ${\rm{Gal}}(K_f^\proj/\Qq)$ is conjugate to $G$, the field $K_f^\proj$ contains $L$. Now, combine (a), (b) and Lemma \ref{kln2} to get the existence of a $\Qq$-regular Galois extension of $\Qq(T)$ of Galois group $G$ which specialises to $\Qq(\sqrt{-1})/\Qq$ at a non branch point. By Proposition \ref{pv}, we then get a finite Galois extension $M_1/\Qq$ of group $G$ which is totally imaginary. Denote the prime numbers which ramify in $M_1/\Qq$ by $p_1, \dots, p_s$. Apply again (a), (b) and Lemma \ref{kln2} to get a $\Qq$-regular Galois extension of $\Qq(T)$ of Galois group $G$ which specialises to $\Qq(\sqrt{-1})/\Qq$ and to $\Qq/\Qq$ at non branch points. One then gets a finite Galois extension $M_2/\Qq$ of group $G$ which is totally imaginary and unramified at $p_1, \dots, p_s$. In particular, the fields $M_1$ and $M_2$ are linearly disjoint over $\Qq$. But, by Proposition \ref{tool0}, there exist two normalised Hecke eigenforms $f_1$ and $f_2$ such that $M_i= K_{f_i}^\proj$ for $i=1,2$, thus leading to a contradiction. Hence, $E/\Qq(T)$ is $\Qq$-regular. 

Next, by (c), the group $G$ has a non-cyclic abelian subgroup $H$. Without loss of generality, we may assume $H= \Zz/p_0 \Zz \times \Zz/p_0 \Zz$ for some prime number $p_0$. Pick a sufficiently large prime number $q$ which is totally split in the number field $F(e^{2 i \pi / {p_0}})$, where $F$ is the number field provided by Proposition \ref{kln}. As $q$ is totally split in $F$, every specialisation of $E/\Qq(T)$ has cyclic decomposition group at $q$. Hence, for every normalised Hecke eigenform $f$ such that ${\rm{Gal}}(K_f^\proj/\Qq)$ is conjugate to $G$, the field $K_f^\proj$ has cyclic decomposition group at $q$. However, since $q$ is totally split in  $\Qq(e^{2 i \pi / {p_0}})$, one has $q \equiv 1 \pmod {p_0}$ (up to finitely many exceptions) and there exists a Galois extension $F^{(q)}$ of $\Qq_q$ of group $\Zz/p_0\Zz \times \Zz/p_0\Zz$. \footnote{Indeed, one can take $F^{(q)}$ to be the compositum of the fields $F_1^{(q)}$ and $F_2^{(q)}$, where $F_1^{(q)}$ is the unique degree $p_0$ unramified extension of $\Qq_q$ and $F_2^{(q)}/\Qq_q$ is a finite Galois extension with Galois group $\Zz/p_0\Zz$ that is totally ramified (such an extension exists; see, e.g., \cite[Chapter IV]{Ser79}).} Now, by \cite[(9.2.8)]{NSW08}, there exists a Galois extension $F/\Qq$ of group $\Zz/p_0\Zz \times \Zz/p_0\Zz$ whose completion at $q$ is equal to $F^{(q)}/\Qq_q$. Consequently, by (c), Lemma \ref{kln2} and Proposition \ref{pv}, we get a finite Galois extension $M/\Qq$ of group $G$, which is totally imaginary and such that the completion at $q$ has Galois group $\Zz/p_0\Zz \times \Zz/p_0\Zz$. By Proposition \ref{tool0}, we get that $M=K_f^\proj$ for some normalised Hecke eigenform $f$, thus leading to another contradiction.

\vspace{0.5mm}

\noindent
(2) The proof is identical to that of (1).
\end{proof}

\section{Infinite Galois families of non-liftable weight one modular eigenforms} \label{sec:explicit}

The aim of this section is to exhibit an infinite regular $1$-parameter projective Galois family of non-liftable Katz modular eigenforms of weight one over ${\overline{\Ff_p}}$ for $p \in \{3,5,7,11\}$.

We start with a general result, which potentially applies to any odd prime number~$p$ \footnote{The oddness of $p$ is only needed because the weight lowering result used in the proof of Theorem \ref{thm1} does not have any published proof in the literature when $p=2$, the representation is unramified at $p=2$ and the image of Frobenius at $p=2$ is scalar. However, a proof is outlined on Frank Calegaris's blog (see \url{https://www.galoisrepresentations.com/2014/08/10/is-serres-conjecture-still-open/}), making the restriction $p>2$ superfluous.}. Consider the following statement:

\vspace{1.5mm}

\noindent
{\rm{($*$)}} {\it{Let $p$ be an odd prime number, $n$ a positive integer and let $G$ be either ${\rm{PGL}}_2(\mathbb{F}_{p^n})$ or ${\rm{PSL}}_2(\mathbb{F}_{p^n})$. There exists an infinite regular $1$-parameter projective $G$-Galois family consisting of Katz modular forms of weight one. Moreover, no family member is liftable to a holomorphic weight one Hecke eigenform in any level.}}

\begin{thm} \label{thm1}
Statement {\rm{($*$)}} holds if $G$ is not isomorphic to any finite subgroup of ${\rm{PGL}}_2(\mathbb{C})$ and if there exists a $\Qq$-regular Galois extension $E/\Qq(T)$ of group $G$ such that the following two conditions hold:

\vspace{0.5mm}

\noindent
{\rm{(1)}} there exists $t_0 \in \Qq$, not a branch point of $E/\mathbb{Q}(T)$, such that $E_{t_0}/\Qq$ is totally imaginary,

\vspace{0.5mm}

\noindent
{\rm{(2)}} there exists $t_0 \in \Qq$, not a branch point of $E/\mathbb{Q}(T)$,  such that $E_{t_0}/\Qq$ is unramified at $p$.
\end{thm}

\begin{proof}
By the second part of the assumption and Proposition \ref{pv}, the extension $E/\Qq(T)$ has infinitely many distinct specialisations of group $G$ which are totally imaginary and unramified at~$p$. We view any such specialisation $E_{t}/\Qq$ as a projective Galois representation
$$\rho_{t}^\proj: \GQ \twoheadrightarrow \Gal(E_{t}/\QQ) \cong G \subset {\rm{PGL}}_2(\overline{\Ff_p}),$$
thus obtain infinitely non-isomorphic ones. By the result of Tate recalled in \S\ref{ssec:gr}, there is a linear lift
$\rho_{t}: \GQ \to {\rm{GL}}_2(\overline{\Ff_p})$
of $\rho_{t}^\proj$ which is unramified at all prime numbers where $E_{t}/\QQ$ is unramified. Moreover, as in the proof of Proposition \ref{tool0}, $\rho_{t}$ is odd, as $E_{t}$ is totally imaginary. As it is also irreducible, by Theorem~\ref{thm:Serre-Conj}, the representation $\rho_{t}$ comes from some normalised Hecke eigenform. Furthermore, by weight lowering as proved in \cite[Theorem 4.5]{Edi92}, $\rho_{t}$ actually comes from a Katz modular form $f_{t}$ of weight $1$ over~$\overline{\Ff_p}$.
In order to see this, note that we are in case 2.(a) in \cite[Definition 4.3]{Edi92} with $a=b=0$, whence the weight associated with $\rho_t$ equals~$1$. Moreover, note that the hypothesis excluding the `exceptional case' in \cite[Theorem 4.5]{Edi92} is superfluous by the last sentence of \cite[\S1]{Edi92}.

Finally, $G$ is not isomorphic to a quotient of any finite subgroup of $\PGL_2(\CC)$. Indeed, otherwise one would have that $G$ is a quotient of a cyclic group or of a dihedral group or of a finite group among $A_4$, $S_4$ and $A_5$. As this family of groups is easily seen to be closed under quotients, one would have that $G$ itself is cyclic, dihedral or among $A_4$, $S_4$ and $A_5$, which cannot happen by the first part of the assumption. Consequently, the representation $\rho_{t}$ cannot be the reduction of any semi-simple $2$-dimensional Artin representation. Hence, $f_{t}$ cannot be the reduction of a normalised holomorphic Hecke eigenform of weight $1$ and any level.
\end{proof}

Now, we combine Theorem \ref{thm1} and the various tools from \S\ref{ssec:ffe} to show that Statement ($*$) holds for $p \in \{5,7,11\}$:

\begin{cor} \label{coro4}
{\rm{(1)}} Statement {\rm{($*$)}} holds for $p=5$ (with $G= {\rm{PGL}}_2(\Ff_5) \cong S_5$).

\vspace{0.5mm}

\noindent
{\rm{(2)}} Statement {\rm{($*$)}} holds for $p=7$ (with $G={\rm{PSL}}_2(\Ff_7)$).

\vspace{0.5mm}

\noindent
{\rm{(3)}} Statement {\rm{($*$)}} holds for $p=11$ (with $G={\rm{PSL}}_2(\Ff_{11})$).
\end{cor}

\begin{proof}
(1) Consider the monic separable polynomial $P(T,Y) = Y^5-Y^4-T$ and denote its splitting field over $\Qq(T)$ by $E$. By \cite[\S4.4]{Ser92}, the extension $E/\Qq(T)$ is $\Qq$-regular, has $r=3$ branch points and has Galois group $S_5 \cong {\rm{PGL}}_2(\mathbb{F}_5)$, which is not a subgroup of ${\rm{PGL}}_2(\Cc)$. Then, by Proposition \ref{df}(1), for every rational number $t_0$, the specialisation $E_{t_0}/\Qq$ is not totally real as soon as $t_0$ is not a branch point. Moreover, by, e.g., \cite[Theorem 2]{Swa62}, the discriminant of $P(T,Y)$ is equal to $5^5T^4+4^4 T^3$, which is not in $5 \Zz[T]$. Hence, the extension $E/\Qq(T)$ has no vertical ramification at $p=5$, by Proposition \ref{nvr}. Furthermore, since $p =5 \geq 4=r+1$, we may use Lemma \ref{lemma3} to get the existence of $t_0 \in \Qq$ such that $t_0$ does not meet any branch point of $E/\Qq(T)$ modulo $p=5$. Hence, by Proposition \ref{sit}, there exists $t_0 \in \Qq$, not a branch point of $E/\mathbb{Q}(T)$,  such that $E_{t_0}/\Qq$ is unramified at $p=5$. It then remains to apply Theorem \ref{thm1} to conclude.

\vspace{0.5mm}

\noindent
(2) The proof is similar in the case $p=7$. Namely, consider the monic separable polynomial
$$P(T,Y)=Y^7 - 56 Y^6 + 609 Y^5 + 1190 Y^4 + 6356 Y^3 + 4536 Y^2 - 6804 Y - 5832 -TY(Y+1)^3$$
and denote its splitting field over $\Qq(T)$ by $E$. By \cite[Satz 3]{MM85}, the extension $E/\Qq(T)$ is $\Qq$-regular, has three branch points and has Galois group ${\rm{PSL}}_2(\mathbb{F}_7)$, which is not contained in ${\rm{PGL}}_2(\Cc)$. Moreover, the reduction modulo $p=7$ of $P(T,Y)$ is $Y^7 - 1 - TY(Y+1)^3$, which has discriminant $-3T^8-T^7 \not= 0$. As above, we apply the various tools from \S\ref{ssec:ffe} and Theorem \ref{thm1} to get the desired conclusion.

\vspace{0.5mm}

\noindent
(3) Consider the monic separable polynomial
$$P(T,Y)= Y^{11} -3Y^{10} + 7Y^9 - 25Y^8 + 46Y^7 - 36Y^6 + 60Y^4 -121Y^3 + 140Y^2-95Y+27 + Y^2(Y-1)^3T$$
and denote its splitting field over $\Qq(T)$ by $E$. The extension $E/\Qq(T)$ is $\Qq$-regular and has Galois group ${\rm{PSL}}_2(\Ff_{11})$, which does not embed into ${\rm{PGL}}_2(\Cc)$ (see page 497 of \cite{MM18} for more details). Moreover, one checks with a computer that the discriminant $\Delta(T)$ of $P(T,Y)$ is
$$\Delta(T) = (108T^3 - 7472T^2 + 267408T + 7987117)^4.$$
Hence, $E/\Qq(T)$ has at most 4 branch points and one may then apply Lemma \ref{lemma3} to get that there exists $t_0 \in \Qq$ such that $t_0$ does not meet any branch point modulo $p=11$. Also, as $\Delta(T)$ is not in $11\Zz[T]$, the extension $E/\Qq(T)$ has no vertical ramification at $p=11$, by Proposition \ref{nvr}. Hence, by Proposition \ref{sit}, there exists $t_0 \in \Qq$, not a branch point of $E/\mathbb{Q}(T)$,  such that $E_{t_0}/\Qq$ is unramified at $p=11$. Concerning the local behaviour at the infinite prime, it actually holds that $E/\Qq(T)$ has four branch points\footnote{Indeed, as recalled in \S\ref{ssec:ffe}, every branch point of $E/\Qq(T)$ is either $\infty$ or a root of $\Delta(T)$. By, e.g., \cite[Lemma 3.1]{Mue02}, $\infty$ is a branch point of $E/\Qq(T)$  (the corresponding ramification index is even equal to 6). Moreover, by the Riemann existence theorem, at least one root of $\Delta(T)$ is a branch point of $E/\Qq(T)$. As $108T^3 - 7472T^2 + 267408T + 7987117$ is irreducible over $\Qq$ (it is easily checked that its reduction modulo 5 has no root in $\Ff_5$), all roots of $\Delta(T)$ have to be branch points of $E/\Qq(T)$, by the so-called Branch Cycle Lemma (see \cite{Fri77} and \cite[Lemma 2.8]{Vol96}).} and, because of that, we cannot use Proposition \ref{df}(1) as above. We then refer to Proposition \ref{df}(2). Namely, the 9-th derivative with respect to $Y$ of $P(T,Y)$ is
$$\frac{11!}{2}Y^2 - 3 \cdot 10! \cdot Y +  7 \cdot 9!,$$
which has discriminant $10! \cdot 9! \cdot (9 \cdot 10-2 \cdot 11 \cdot 7) <0$. Hence, the specialisation $E_{t_0}/\Qq$ is not totally real for every rational number $t_0$ such that $P(t_0,Y)$ is separable. As in the previous cases, it then remains to apply Theorem \ref{thm1} to conclude the proof.
\end{proof}

\begin{rem}
Of course, variants can be given, by making use of other explicit polynomials $P(T,Y)$. For example, in the case $p=7$, one can also use the polynomial 
$$P(T,Y)= Y^7+ Y^6 + Y^5 + T Y^4 + (T-2) Y^3 - 5 Y^2 - 2Y +1 \in \QQ[T][Y],$$
which is intensively studied in \cite{LSY12}, to prove that Statement {\rm{($*$)}} holds (with $G={\rm{PSL}}_2(\mathbb{F}_7)$ too).
\end{rem}

A tool used throughout the proof of Corollary \ref{coro4} is Lemma \ref{lemma3}, which does not apply in the case $p=3$ if ${\rm{Gal}}(E/\Qq(T))$ is not cyclic (by the Riemann existence theorem). However, Statement ($*$) still holds in this case:

\begin{cor} \label{cor5}
Statement {\rm{($*$)}} holds for $p=3$ (with $G= {\rm{PSL}}_2(\Ff_9) \cong A_6$).
\end{cor}

\begin{proof}
Clearly, $A_6 \cong {\rm{PSL}}_2(\Ff_9)$ does not embed into ${\rm{PGL}}_2(\Cc)$. Now, consider the polynomial $f(Y)=(Y^2+1)(Y^2+4) \in \Qq[Y]$. It is separable and has discriminant $4 \cdot 9 \cdot 9 \cdot 16$, which is a square in $\Qq$. The splitting field over $\Qq$ of this polynomial is $\Qq(\sqrt{-1})$, which is unramified at 3 and totally imaginary. By \cite[Theorem 3]{KM01}, there exists a monic separable polynomial $P(T,Y) \in \Qq[T][Y]$ of splitting field $E$ over $\Qq(T)$ such that $E/\Qq(T)$ is a $\Qq$-regular Galois extension of group $A_6$ and such that the splitting fields over $\Qq$ of $P(0,Y)$ and $f(Y)$ coincide. From this equality and the fact that $f(Y)$ is separable, we get that the specialised field $E_0$ is equal to the splitting field of $f(Y)$ over $\Qq$ and that $0$ is not a branch point of $E/\mathbb{Q}(T)$. It then remains to apply Theorem \ref{thm1} to conclude.
\end{proof}

Finally, we discuss the case $p \geq 13$. Another common feature of the proof of Corollary \ref{coro4} is the existence of a monic separable polynomial $P(T,Y) \in \Zz[T][Y]$ of discriminant $\Delta(T) \not \in p\Zz[T]$ and of Galois group ${\rm{PGL}}_2(\mathbb{F}_{p^n})$ or ${\rm{PSL}}_2(\mathbb{F}_{p^n})$ over $\Qq(T)$ (for some $n \geq 1$). For $p \geq 13$, we are not aware of any polynomial satisfying both conditions. For example, no explicit polynomial of group ${\rm{PGL}}_2(\mathbb{F}_p)$ ($11 \leq p \leq 29$) given in pages 499-500 of \cite{MM18} satisfies the former. We also notice that Statement ($*$) for the given prime number $p$ implies that some ${\rm{PGL}}_2(\Ff_{p^n})$ or ${\rm{PSL}}_2(\Ff_{p^n})$ occurs as the Galois group of a $\Qq$-regular Galois extension of $\Qq(T)$, which is unknown in general. Of course, for some prime numbers $p$, this is known (usually for $n=1$) and one even has such $\Qq$-regular extensions with three branch points, coming from the rigidity method (see, e.g., \cite[Chapter I, Corollary 8.10]{MM18} for more details). In particular, for such a prime number $p$, we obtain an infinite regular $1$-parameter projective $G$-Galois family, with $G={\rm{PGL}}_2(\Ff_p)$ or $G={\rm{PSL}}_2(\Ff_p)$.

Nevertheless, one has the following result:

\begin{prop} 
Let $p$ be an odd prime number. Then there exist a finite group $G$ of order $2p^2$ and a $\Qq$-regular Galois extension $E/\Qq(T)$ of group $G$ such that the following conditions hold:

\vspace{0.5mm}

\noindent
{\rm{(1)}} $G \subseteq {\rm{PGL}}_2(\overline{\mathbb{F}_p})$ but $G \not \subseteq {\rm{PGL}}_2(\Cc)$,

\vspace{0.5mm}

\noindent
{\rm{(2)}} there exists $t_0 \in \Qq$, not a branch point of $E/\mathbb{Q}(T)$,  such that $E_{t_0}/\Qq$ is totally imaginary,

\vspace{0.5mm}

\noindent
{\rm{(3)}} there exists $t_0 \in \Qq$, not a branch point of $E/\mathbb{Q}(T)$,  such that $E_{t_0}/\Qq$ is unramified at $p$.
\end{prop}

\begin{proof}
Consider the subset 
$$G=\bigg\{\begin{pmatrix} 
 1 & x \\
 0 & 1
 \end{pmatrix},
\begin{pmatrix} 
 -1 & x \\
 0 & 1
 \end{pmatrix}
\, \, : \, \, x \in \mathbb{F}_{p^2} \bigg\}$$
of ${\rm{PGL}}_2(\mathbb{F}_{p^2}) \subset {\rm{PGL}}_2(\overline{\Ff_p})$. It is easily checked that $G$ is a subgroup of ${\rm{PGL}}_2(\mathbb{F}_{p^2})$ of order $2p^2$. Actually, one has
\begin{equation} \label{eq}
G \cong (\Zz/p\Zz \times \Zz/p\Zz) \rtimes \Zz/2\Zz.
\end{equation}
Moreover, the group $G$ is not a subgroup of ${\rm{PGL}}_2(\Cc)$ (hence, (1) holds). Indeed, one cannot have $G \cong S_4, A_4, A_5$ for cardinality reasons. Moreover, if $G$ was either cyclic or dihedral, then its unique $p$-Sylow subgroup would be $\Zz/p^2\Zz$, which cannot happen.

Now, set $F_1/\Qq= \Qq(\sqrt{-1})/\Qq $ and $F_2/\Qq=\Qq/\Qq$. By, e.g., \cite[Theorem 0.5.3]{JLY02}, $\Zz/p\Zz \times \Zz/p\Zz$ has a generic polynomial with rational coefficients. Clearly, the same is also true for $\Zz/2\Zz$. Consequently, by \eqref{eq} and a well-known result of Saltman (see, e.g., \cite[Corollary 7.2.2]{JLY02}), $G$ has a generic polynomial over $\Qq$. It then remains to apply Lemma \ref{kln2} to construct a $\Qq$-regular Galois extension $E/\Qq(T)$ of group $G$ which specializes to $F_1/\Qq$ and $F_2/\Qq$ at non branch points, thus ending the proof.
\end{proof}

Unfortunately, this result does not apply in the same way as for the previous examples because the group $G$ does not occur as the image of a $2$-dimensional semi-simple representation over $\overline{\Ff_p}$, hence, we cannot immediately get modularity results.

\bibliography{Biblio2}
\bibliographystyle{alpha}

\bigskip

\noindent 
\begin{tabular}{lll}
Sara Arias-de-Reyna   & Fran\c{c}ois Legrand    & Gabor Wiese\\
Departamento de \'Algebra  & Institut f\"ur Algebra & Department of Mathematics\\
Facultad de Matem\'aticas  & Fachrichtung Mathematik & \\
Universidad de Sevilla     & Technische Universit\"at Dresden & University of Luxembourg \\
Spain                      & Germany & Luxembourg \\
\url{sara_arias@us.es} & \url{francois.legrand@tu-dresden.de}    & \url{gabor.wiese@uni.lu}\\
\end{tabular}

\end{document}